\documentclass[reqno, final, a4paper]{amsart}
\usepackage{color}
\usepackage{amsmath, amssymb, amsthm}
\usepackage{mathrsfs}

\usepackage{mathtools}

\usepackage{fullpage}
\usepackage[noadjust]{cite}
\usepackage{graphics}
\usepackage{pifont}
\usepackage{booktabs}
\usepackage{multirow}
\usepackage{cancel}
\usepackage{tikz}
\usepackage{bbm}
\usepackage[T1]{fontenc}
\usetikzlibrary{arrows.meta}

\usepackage{environ}
\usepackage{paralist}
\usepackage{url}

\usepackage[labelfont=bf]{caption}
\usepackage{framed}
\usepackage[framemethod=tikz]{mdframed}
\usepackage{pgfplots}
\usepackage{appendix}
\usepackage{graphicx}
\usepackage[textsize=tiny]{todonotes}
\usepackage{tcolorbox}
\usepackage{xcolor}
\usepackage[shortlabels]{enumitem}
\allowdisplaybreaks

\usepackage{euflag}
\usepackage{cases}

\usepackage[mathlines]{lineno}
\usepackage{etoolbox} 
\newcommand*\linenomathpatch[1]{
   \expandafter\pretocmd\csname #1\endcsname {\linenomath}{}{}
   \expandafter\pretocmd\csname #1*\endcsname{\linenomath}{}{}
   \expandafter\apptocmd\csname end#1\endcsname {\endlinenomath}{}{}
   \expandafter\apptocmd\csname end#1*\endcsname{\endlinenomath}{}{}
 }
\newcommand*\linenomathpatchAMS[1]{%
    \expandafter\pretocmd\csname #1\endcsname {\linenomathAMS}{}{}%
    \expandafter\pretocmd\csname #1*\endcsname{\linenomathAMS}{}{}%
    \expandafter\apptocmd\csname end#1\endcsname {\endlinenomath}{}{}%
    \expandafter\apptocmd\csname end#1*\endcsname{\endlinenomath}{}{}%
}

\expandafter\ifx\linenomath\linenomathWithnumbers
\let\linenomathAMS\linenomathWithnumbers
\patchcmd\linenomathAMS{\advance\postdisplaypenalty\linenopenalty}{}{}{}
\else
\let\linenomathAMS\linenomathNonumbers
\fi

\linenomathpatchAMS{gather}
\linenomathpatchAMS{multline}
\linenomathpatchAMS{align}
\linenomathpatchAMS{alignat}
\linenomathpatchAMS{flalign}
\linenomathpatch{equation}
\linenomathpatch{numcases}
\linenomathpatch{figure}

\let\:=\colon
\let\phi=\varphi

\newcommand{\cA}{\mathcal{A}}

\newcommand{\cS}{\mathcal{S}}

\def\NN{{\mathbb N}}

\let\epsilon\varepsilon

\usepackage{graphicx} 
\usepackage{algpseudocode,algorithm,multicol}
\usepackage{setspace}

\newtheorem{theorem}             {Theorem}[section]
\newtheorem{lemma}   [theorem]   {Lemma}        
\newtheorem{conjecture}  [theorem] {Conjecture}   
   
\newtheorem{proposition}  [theorem] {Proposition}   
\newtheorem{corollary}  [theorem] {Corollary}

\newtheorem{proto-theo}[theorem] {Proto-Theorem}   

\theoremstyle{definition}
\newtheorem{definition}  [theorem] {Definition}   
\newtheorem{innercustomthm}{Case}
\newenvironment{case}[1]
  {\renewcommand\theinnercustomthm{#1}\innercustomthm}
  {\endinnercustomthm}
\theoremstyle{remark}
\newtheorem{observation}[theorem] {Observation}   
\newtheorem{remark}  [theorem] {Remark}   

\pgfplotsset{compat=1.18}

\usepackage[colorlinks=true,allcolors=blue
]{hyperref}
 
\begin{document}

\onehalfspace

\title[Monochromatic products in random sets]{Monochromatic products in random integer sets}
\author[Roger Lidón]{Roger Lidón}
\address{Departament de Matem\`atiques, Universitat Polit\`ecnica de Catalunya (UPC), Carrer de Pau Gargallo 14, 08028  Barcelona, Spain.}
\email{\{rlidon2006\,|\,dariogol.com\,|\,pmorrismaths\,|\,miquel.ortega9\}@gmail.com}
\author[Darío Martínez] {Darío Martínez}
\author[Patrick Morris]{Patrick Morris}
\author[Miquel Ortega]{Miquel Ortega}

\thanks{P. Morris was supported  by the European Union's Horizon Europe   Marie Sk{\l}odowska-Curie grant RAND-COMB-DESIGN - project number
101106032 {\euflag}. M.Ortega was supported by the project PID2023-147202NB-I00, funded by MICIU/AEI/10.13039/501100011033/, as well as the FPI grant PRE2021-099120.}
\date{\today}

\begin{abstract}
A well-known consequence of Schur's theorem is that for $r\in \NN$, if $n$ is sufficiently large, then any $r$-colouring of $[n]$ results in monochromatic $a,b,c\in [n]$ such that $ab=c$. In this paper we are interested in the threshold at which the binomial random set $[n]_p$ almost surely inherits this Ramsey-type property. In particular for $r=2$ colours,  we show that this threshold lies between $n^{-1/9-o(1)}$ and $n^{-1/11}$. Whilst analogous questions for solutions to (sets of) linear equations are now well understood, our work   suggests that both the behaviour of the thresholds and the proof methods needed to determine them differ substantially in the  non-linear setting. 
\end{abstract}

\maketitle

\section{Introduction}

A collection of (ordered) subsets $\cS\subset \NN$ is  \textit{partition-regular}\footnote{This definition is often stated in an infinitary fashion, asking for any $r$-colouring of $\NN$ to give monochromatic $S\in \cS$. When every member $S\in \cS$ is finite, these two formulations are equivalent by the so-called \textit{compactness principle}.} if for any $r\in \NN$,  any $r$-colouring of $[n]:=\{1,2\ldots,n\}$ with  $n\in \NN$ is sufficiently large,  results in a monochromatic set $S\in \cS$. Determining which arithmetic configurations lead to collections of sets  $\cS$ that are partition-regular is a major theme in modern combinatorics. The earliest result in this area of arithmetic Ramsey theory (and indeed in Ramsey theory as a whole) is due to Schur \cite{SchurUpperBound} over a century ago, who showed that the collection of all \textit{sums} $\{(a,b,c)\in \NN^3:a+b=c\}$  is partition-regular. In 1933, Rado \cite{rado1933studien} proved a far-reaching generalisation of Schur's result, characterising for which matrices $A$ the system of linear homogeneous equations $A\mathbf{x}=0$ has a solution set that is partition-regular.  This celebrated result in particular implies  both Schur's theorem and the famous theorem of  van der Waerden  \cite{van1927beweis} showing that for any $r$, the collection of $r$-term arithmetic progressions $\{(a,a+d,\ldots,a+(r-1)d):a,d\in \NN\}$ is partition-regular. Further work completely settled partition-regularity for collections of sets defined by finite sets of 
 equations, see e.g.\ the survey of Hindman \cite{Hindman_survey}. 

\subsection{Random sets of integers}  
In the 1990s, researchers started to strengthen statements showing partition-regularity  by exploring the Ramsey properties of \textit{sparse} host sets. In the context of monochromatic sums,   we define a set $A\subseteq \NN$ to be \textit{$r$-Schur} if any $r$-colouring of $A$ results in  monochromatic  $(a,b,c)\in A^3$ with $a+b=c$. When $r=2$, we simply refer to the set $A$ as being 
\emph{Schur}. Whilst Schur's theorem \cite{SchurUpperBound} guarantees that $[n]$ is Schur for all $n$ sufficiently large, it turns out that much sparser subsets of $[n]$ are typically also Schur. This was shown in pioneering work of Graham, R\"odl and Ruci\'nski \cite{SumThreshold}. The set $[n]_p$ denotes the binomial random subset of $[n]$ obtained by keeping each integer in $[n]$ with probability $p=p(n)$ independently and we say that a property $\cA\subseteq 2^{[n]}$ holds in $[n]_p$ asymptotically almost surely (a.a.s.\ for short) if the probability that it holds tends to 1 as $n$ tends to infinity. 

\begin{theorem}[\cite{SumThreshold}]
\label{thm:sum threshold} We have that
  \setcounter{equation}{-1}
  \begin{linenomath}
 \begin{numcases}
 {\lim_{n \to \infty} \Pr([n]_p  \mbox{ is  Schur}) =}
	   0 &if  $p=o(n^{-1/2})$;\\
	     1  &if $p= \omega(n^{-1/2}).$
 \end{numcases}
    \end{linenomath}
\end{theorem}


For any property $\cA\subseteq 2^{[n]}$, a function $p_{*}=p_*(n)$ is the\footnote{This is a slight abuse of notation as thresholds are not unique but rather determined up to constant factors.} \textit{threshold} for $\cA$ if whenever $p=\omega(p_*)$, we have that $[n]_p$ a.a.s.\ has the property $\cA$ whilst whenever $p=o(p_*)$, we have that $[n]_p$ a.a.s.\ does \textit{not} have the property $\cA$. Thus Theorem \ref{thm:sum threshold} establishes that $n^{-1/2}$ is the threshold for the Schur property. Further work established  a sparse random analogue of Rado's theorem \cite{rado1933studien}, determining the 
threshold for $[n]_p$ to contain monochromatic solutions  to sets of homogeneous equations $A{\bf{x}}=0$ in any colouring, for any $A$ such that the solution set is partition-regular. The $(0)$-statements  for these thresholds 
 as well as  $(1)$-statements for  a subclass of  so-called \textit{density-regular} matrices $A$  were proven by R\"odl and Ruci\'nski \cite{RadoPartition} who conjectured that their $(0)$-statements provided the correct threshold in all cases.  This took over a decade to verify with Friedgut, R\"odl and Schacht \cite{friedgut2010ramsey} finally providing a  proof for all $(1)$-statements.  The $(1)$-statements for all these Rado-type properties  in fact hold for any bounded number of colours $r$ and, perhaps surprisingly, the thresholds do not depend on the number of colours $r$. In particular, it follows from \cite{friedgut2010ramsey} that for \emph{any} $r\geq 2$, the threshold for $[n]_p$ to be $r$-Schur is at $n^{1/2}$. 
 
 Shortly after the breakthrough work of Friedgut, R\"odl and Schacht \cite{friedgut2010ramsey} providing $(1)$-statements for the random Rado theorem,  Conlon and Gowers \cite{2016.CG} and independently Schacht \cite{2016.Schacht} developed powerful \textit{transference principles} which establish optimal $(1)$-statements for a wide range of extremal and Ramsey properties in random discrete structures. In particular, these provided new proofs for some of  $(1)$-statements for Rado-type problems discussed above. By adopting an abstract viewpoint of independent sets in hypergraphs, Balogh, Morris and Samotij \cite{2015.BMS} and independently Saxton and Thomason \cite{2015.ST} then developed the theory of \textit{hypergraph containers} which encapsulated many of the applications of the previous transference principles and led to an extremely flexible and elegant tool for establishing $(1)$-statements in random discrete structures. In particular, Nenadov and Steger \cite{nenadov2016short} showed that one can use hypergraph containers to establish thresholds for Ramsey properties. This has led to a golden age in random Ramsey theory with many recent advances. Restricting only to the arithmetic setting here,  we mention thresholds for solutions with repeated entries to Rado-type equations  \cite{spiegel2017note}, resilience versions of Rado thresholds \cite{hancock2019independent}, perturbed Schur thresholds  \cite{aigner2019monochromatic,das2024schur}, asymmetric random Rado theorems \cite{aigner2023asymmetric,hancock2022asymmetric}, sharp thresholds for Schur and van der Waerden  \cite{friedgut2015sharp,friedgut2022sharp,schulenburg2016threshold}, thresholds for Rado type properties in random subsets of the primes and abelian groups \cite{freschi2024typical} and canonical van der Waerden thresholds \cite{alvarado2024canonical}.

\subsection{Monochromatic products} All of the random arithmetic Ramsey results mentioned above concern solutions to linear equations. In this paper, we make a first foray into exploring non-linear equations by considering monochromatic products. The fact that the set\footnote{We omit the integer 1 here to avoid trivialities.} $\{(a,b,c)\in ([n]\setminus\{1\})^3:ab=c\}$ is partition-regular can be derived from Schur's theorem for sums by considering powers of 2. In general, we say a set $A\subseteq \NN$ is $r$-product-Schur if any $r$-colouring of $S$ results in monochromatic $(a,b,c)\in A^3$ with $ab=c$ and when $r=2$, we will just say that $S$ is product-Schur. As the property of being product-Schur is monotone increasing\footnote{That is, if $A$ is product-Schur and $A\subseteq B$ then $B$ is also product-Schur.}, the Bollob\'as-Thomason theorem \cite{bollobas1987threshold} implies that there exists a threshold for the property. Our main theorem gives bounds on this threshold. 

\begin{theorem} \label{thm: 2 col}
    We have that
  \setcounter{equation}{-1}
  \begin{linenomath}
 \begin{numcases}
 {\lim_{n \to \infty} \Pr([n]_p  \mbox{ is  product-Schur} ) =}
	   0 &if $p\leq n^{-1/9-o(1)} $;\\
	     1  &if $p\geq\omega(n^{-1/11})$.
 \end{numcases}
    \end{linenomath}
\end{theorem}

\begin{remark} \label{rem:degenerate}
    We remark that in our definition of being product-Schur we allow for \textit{degenerate} products $a^2=b$. Our (1)-statement in Theorem \ref{thm: 2 col} works in the slightly stronger setting of forcing non-degenerate monochromatic products. That is, if $p=\omega(n^{-1/11}$) then a.a.s.\ any $2$-colouring of $[n]_p$ will result in monochromatic $a,b,c\in [n]_p$ with $ab=c$ and $a\neq b$. 
\end{remark}

Theorem \ref{thm: 2 col} makes a first step towards answering a recent question of Mattos, Mergoni Cecchelli and Parczyk \cite{mattos2025product}. Their motivation came from a question of Prendiville at the 2022 British Combinatorial Conference (BCC), which asked for another strengthening of Schur's theorem to be generalised to the setting of products.  Answering a $\$100$ question of Graham \cite{SumThreshold}, Robertson and Zeilberger \cite{robertson19982} and independently Schoen \cite{schoen1999number} proved that any 2-colouring of $[n]$ gives $(\tfrac{1}{11}+o(1))n^2$ monochromatic sums  and that this is tight (their results are listed with the constant $1/22$ as they consider $(a,b,c)$ and $(b,a,c)$ as the same solution).  For more colours, determining asymptotically the minimum number of monochromatic sums is an intriguing open question. Prendiville asked how many monochromatic products are necessarily guaranteed when colouring $[n]\setminus \{1\}$. This was then addressed by Arag\~ao, Chapman, Ortega and Souza \cite{aragao2023number} and independently by  Mattos, Mergoni Cecchelli and Parczyk \cite{mattos2025product}. For two colours, the  first group \cite{aragao2023number} determined the minimum number of products precisely, showing that it is $(\tfrac{1}{2\sqrt{2}}+o(1))n^{1/2}\log n$. The second group \cite{mattos2025product} showed a lower bound of $n^{1/3-o(1)}$ for this problem and also initiated the study of other variations of Schur's theorem in the product setting. In particular, they determined precisely the threshold for the appearance of a product in the random set $[n]_p$ and asked the question of determining the threshold for the random set to be product-Schur. 

The collection of products in $[n]\setminus\{1\}$ behaves very differently to the collection of sums. Indeed there are $\Theta(n\log n)$ products which is much less than the $\Theta(n^2)$ sums in $[n]$, and the collection is less regular, with the number of products involving an integer $a\in [n]\setminus \{1\}$ depending heavily on the location of $a$ in the interval. 
In the Ramsey setting, the so-called \textit{Schur multiplicity} results discussed above also show a difference in behaviour between the sum and product case. Indeed for sums when $n$ is large we have that a positive fraction of all sums are monochromatic for any 2-colouring of $[n]$,  an example of a phenomenon known as \textit{supersaturation}. For products on the other hand, there are colourings in which only a vanishing fraction of the $\Theta(n\log n)$ products are monochromatic. Given this, it is perhaps not surprising that the proof of our (1)-statement for Theorem \ref{thm: 2 col} differs from existing proofs in random Ramsey theory. Indeed, the transference principles \cite{2016.CG,2016.Schacht} and the hypergraph container method \cite{2015.BMS,2015.ST} discussed above crucially rely on supersaturation and thus seem to be powerless in the setting of monochromatic products. To prove our (1)-statement, we instead appeal to a certain parameterised family of small configurations which have the product-Schur property and show that a.a.s.\ one of these configurations will appear in our random set $[n]_p$. This is reminiscent of the proof of the $(1)$-statement for Schur's theorem in the randomly perturbed model \cite{aigner2019monochromatic,das2024schur} when starting with a dense set of integers. The configurations we use are also similar (in fact contain) those used in \cite{aragao2023number} to count monochromatic solutions.  Known proof methods establishing $(0)$-statements for random Ramsey results also seem ineffective in the product setting and we establish the (0)-statement of Theorem \ref{thm: 2 col} through a careful greedy colouring procedure. 

\subsection{More colours}  We also give initial bounds on thresholds 
for the random set to be $r$-product-Schur for $r\geq 3$. In particular, these
bounds show that, unlike all of the thresholds  \cite{friedgut2010ramsey} for solutions to sets of linear equations as in Rado's theorem, for monochromatic products, the threshold depends on the number of colours. To state our results, we introduce some definitions. The \textit{Schur number} $S(r)$ is the minimum $n$ such that  any $r$-colouring of $[n]$ results in a monochromatic sum. The \textit{double-sum Schur number} $S'(r)$ is the minimum $n$ such that any $r$-colouring of  $[n]$ results in either  a monochromatic sum $(a,b,c)\in [n]^3$ with $a+b=c$ or a  monochromatic \textit{shifted sum}  $(a,b,c)\in [n]^3$ with $a+b=c-1$. Double-sum Schur numbers were introduced by Abbott and Hanson \cite{AbbottHanson} and also feature in the work of Mattos, Mergoni Cecchelli and Parczyk \cite{mattos2025product} bounding  the size of the largest $r$-product-Schur subset of $[n]$. 

\begin{theorem} \label{thm: more col}
    For any $2\leq r\in \NN$ we have that
  \setcounter{equation}{-1}
  \begin{linenomath}
 \begin{numcases}
 {\lim_{n \to \infty} \Pr([n]_p  \mbox{ is  $r$-product-Schur} ) =}
	  0 &if $p=o(n^{-1/S'(r)}) $;\\
	  1  &if $p=\omega(n^{-1/S(r)^2})$.
 \end{numcases}
    \end{linenomath}
\end{theorem}

One can check by hand or by computer that $S'(2)=S(2)=5$ and $S'(3)=S(3)=15$.
The bounds given by Theorem \ref{thm: more col} are thus weaker than those from
Theorem \ref{thm: 2 col} for two colours, whilst for three colours Theorem
\ref{thm: more col} shows that the threshold for being $3$-product-Schur lies
between $n^{-1/14}$ and $n^{-1/196}$, and in particular is larger than the
threshold for being $2$-product Schur. In fact, for a small number of colours
we can do slightly better than the bound given by Theorem \ref{thm: more col}
(0), showing that the threshold for being $3$-product
Schur is at least $n^{-1/19}$, see Remark \ref{rmrk}. In general, the values of $S'(r)$ and $S(r)$ may differ, for example $S'(4)=41<45=S(4)$ but we have that  $3S(r-1)-1\leq S'(r)\leq S(r)$ as shown by Abbott and Hanson \cite[Theorem 4.1]{AbbottHanson}. Calculating (double-sum) Schur numbers is extremely challenging already for $r=5$  \cite{heule2018schur} and the best general bounds give that $\Omega(3.17176^r)\leq S(r)\leq r!e$  with the lower bound due to Exoo \cite{SchurLowerBound} and the upper bound already shown by Schur himself \cite{SchurUpperBound}. Theorem \ref{thm: more col} thus gives that if $p^*_r$ is the threshold for $[n]_p$ being $r$-product Schur, then $n^{-1/c_1^r}\leq p_r^*\leq n^{-1/c_2^{r\log r}}$ for some constants $c_1,c_2>0$.  

\subsection*{Organisation} We prove the (0)-statement of Theorem \ref{thm: 2 col} in Section \ref{sec:0-statement} and the (1)-statement in Section \ref{sec:1-statement}. We then prove Theorem \ref{thm: more col} in Section \ref{sec:more cols} and give some concluding remarks in Section \ref{sec:conclude}. 

\subsection*{Notation}
 We write $[n] \coloneq \{1, \dots, n\}$ and $(m, n] \coloneq \{m+1, \dots, n\}$ for $m ,n \in \mathbb{N}$. Given two sets $A, B \subset [n]$, we write $A \cdot B = \{ab \colon a \in A, b \in B\}$ for the set of pairwise products, and likewise for longer products.  We let $d \colon \mathbb{N} \to \mathbb{N}$ count the number $d(i)$ of divisors of $i \in \mathbb{N}$, and write $D(n) = \max_{1 \leq i \leq n} d(i)$ for the maximum of this function in $[n]$. 
 We say an event occurs asymptotically almost never if it a.a.s.\ does not occur.
 Finally, we omit floors and ceilings when they are not necessary and all logarithms are in the natural base. 

\section{The (0)-statement for two colours} \label{sec:0-statement}
Our goal in this section is to give a proof of the (0)-statement
of Theorem \ref{thm: 2 col}. We first state a more precise version of the
theorem. Define
\[
    \omega(n) \coloneq D(n)^{100},
\]
The precise exponent of
$D(n)$ here is not relevant, we just take it large enough to give us enough
spare room in later bounds. The divisor bound tells us that $D(n) = n^{o(1)}$
(see, for example, Theorem 13.12 in \cite{polylogbound}), so the (1)-statement is a consequence of the following.
\begin{theorem} \label{thm:2_col_0_statement}
    If $p\leq n^{-1/9}\omega(n)^{-1}$, then
a.a.s. we have that $[n]_p$ is not product-Schur.
\end{theorem}
In order to prove Proposition \ref{thm:2_col_0_statement} we must show that, with
high probability, we may partition our random set $S \sim [n]_p$ into two disjoint colours
$B, R$ so that $S = B \cup R$ and $B, R$ are both product-free. We attempt to
build such a colouring using the following greedy algorithm.

\algrenewcommand\algorithmicrequire{\textbf{Input:}}
\begin{algorithm}[H]
\caption{Colouring algorithm for two-coloured case}\label{alg1}
\begin{algorithmic}[1]
    \Require $S \subset \mathbb{Z}, |S| < \infty$
	\State $B \gets \varnothing$
	\State $R \gets \varnothing$
	\For{$k \in S$ in increasing order}
	    \If{$k\not \in (R\cdot R) \cup (R\cdot R\cdot R)$}
	    \State $R \gets R\cup \{k\}$
	    \ElsIf{$k \not \in B\cdot B$}
	    \State $B \gets B\cup \{k\}$
	    \Else
		\State Algorithm fails
	    \EndIf
	\EndFor
\end{algorithmic}
\end{algorithm}

As long as the algorithm does not fail, we obtain such a colouring.
\begin{observation}
    \label{lemma:success_greedy}
    If Algorithm \ref{alg1} terminates successfully on a set $S \subset [n]$,
    then $S$ is not product-Schur.
\end{observation}
\begin{proof}
Since we only add a new element
to $B$ or $R$ if it is not the product of two smaller elements in the set, $B$
and $R$ will be product-free, and if the algorithm successfully terminates we
also have that every integer in $S$ receives a colour.
\end{proof}
Hence, to prove Theorem \ref{thm:2_col_0_statement} it suffices to see that the
algorithm successfully terminates a.a.s.. We now look for
possible causes for the failure of our algorithm, which turn out to be the
following configurations of integers (see Figure
\ref{fig:forbidden_configuration} for a diagram of how such configurations look
like).
\begin{definition}
    \label{def:forbidden_config}
    We say a tuple $M = (a,b,c,d,e,f,x,y,z,u,v,w)$ is a \emph{forbidden
    configuration} in $S$ if it satisfies the following:
    \begin{enumerate}[label=(\roman*)]
        \item The equalities $c = ab = def$, $a = xyz$, and $b = uvw$ hold.
        \item We have that $a,b,c,d,e,x,y,u,v \in S$, and $f,z,w \in S\cup\{1\}$.
        \item $a$ and $b$ are different from $c,d,e,f,x,y,z,u,v,w$.
        \item If we denote $A = \{x,y,z,d,e,f,u,v,w\}\backslash\{1\}$, then $A$ is disjoint to $A\cdot A$ and $A\cdot A\cdot A$.
    \end{enumerate}
\end{definition}
\begin{figure}
    \centering
    \includegraphics[width=0.4\linewidth]{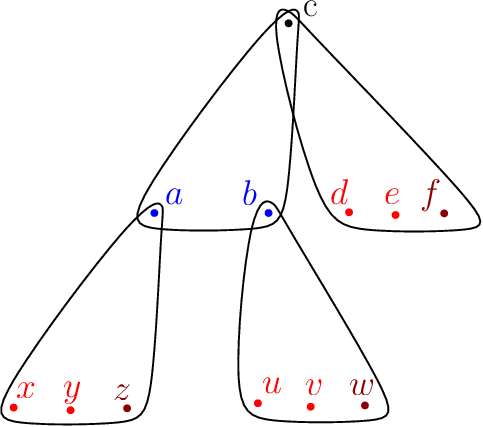}
    \caption{Situation in which Algorithm \ref{alg1} fails. The upper vertex of a triangle is the product of the lower ones, and $z,w,f$ may be either red or one.}
    \label{fig:forbidden_configuration}
\end{figure}
The following lemma tells us that these are the only possible
obstructions for the succesful termination of the algorithm.

\begin{lemma}\label{petaelgreedy}
    If Algorithm \ref{alg1} fails for a set $S$ with $1 \not \in S$, there exists a forbidden configuration
    in $S$.
\end{lemma}
\begin{proof}
    Suppose the algorithm fails to colour $c \in S$. This means that $c$ can be
    written both as $c=ab$ with $a,b\in B$, and as $c = def$ with $d,e\in R$ and
    $f\in R\cup\{1\}$. Notice too that $\{a,b\}\cap \{d,e,f\} =
    \varnothing$, since the algorithm guarantees that $B \cap R = \varnothing$,
    and by assumption $1 \not \in B \subset S$.

    As $a,b\in B$, the definition of the algorithm implies that $a, b \in R\cdot
    R \cup R\cdot R \cdot R$, so there exist $x,y,u,v \in R$, $z,w \in
    R\cup\{1\}$ such that $a = xyz$ and $b=uvw$. Again, each of these six
    numbers must be distinct from $a$ and $b$ because of their different
    colours.

    We claim that the tuple $(a, b, c, d, e, f, x, y, z, u, v, w)$ is a
    forbidden configuration in $S$. In fact, we have already seen properties (i)
    to (iii) in Definition \ref{def:forbidden_config}, and it only remains to
    check that condition (iv) holds. Since in Algorithm $\ref{alg1}$ we only add
    a new element to $R$ if it is not the product of two or three smaller terms
    in $R$, we can guarantee that $R \cap (R \cdot R) = \varnothing$ and $R \cap
    (R\cdot R \cdot R) = \varnothing$. Writing $A = \left\{ x, y, z, d, e, f, u,
    v, w \right\} \setminus \left\{ 1 \right\}$ as in condition (iv), we have
    that $A \subset R$, so $A \cap (A\cdot A) = \varnothing$ and $A \cap (A\cdot
    A \cdot A) = \varnothing$.
\end{proof}

\subsection{Lack of forbidden configurations}
The previous discussion reduces the proof of Theorem \ref{thm:2_col_0_statement}
to the following key proposition, to which we dedicate the rest of the section.

\begin{proposition}
    \label{lem:NoForbidden_v1}
    For   $n\in \NN$,  a function $p=p(n) \leq n^{-1/9}\omega(n)^{-1}$, and $S \sim [n]_p$ we have that a.a.s.\ there is not a forbidden configuration in $S$.
\end{proposition}
\begin{proof}[Proof of Theorem \ref{thm:2_col_0_statement} assuming Proposition
    \ref{lem:NoForbidden_v1}]
    Run Algorithm \ref{alg1} on $S \sim [n]_p$. The algorithm will terminate
    successfully a.a.s.\ on account of Lemma \ref{petaelgreedy},
    Proposition \ref{lem:NoForbidden_v1} and the fact that $\Pr(1 \in S) = p = o(1)$, so the theorem follows from Observation
    \ref{lemma:success_greedy}.
\end{proof}
In fact, we further reduce it to the following, which will allow us to worry
only about forbidden configurations containing large enough elements. 
\begin{proposition}\label{lem:NoForbidden}
    For  $n\in \NN$, $p \leq n^{-1/9}\omega(n)^{-1}$, and $S \sim [n]_p$ we have that a.a.s.\ 
there is not a forbidden configuration in  $S
	\cap (n^{1/9}, n].$
\end{proposition}
\begin{proof}[Proof of Proposition \ref{lem:NoForbidden_v1} assuming Proposition
    \ref{lem:NoForbidden}]
    If there is a forbidden configuration in $S$ that is not a forbidden
    configuration in $S \cap (n^{1/9}, n]$, $S$ must contain a term that is
    smaller than $n^{1/9}$. However, by the union bound we have that
    \[
    \Pr\left(S \cap [n^{1/9}] \neq \varnothing\right) \leq \sum_{s = 0}^{n^{1/9}} 
	\Pr(s \in S) \leq pn^{1/9} \leq \omega(n)^{-1} = o(1).
    \]
    Therefore,
    \begin{multline*}
	\Pr(\text{there is a forbidden configuration in } S) \leq 
	\Pr\left(S \cap [n^{1/9}] \neq \varnothing\right) +\\ \Pr\left(\text{there is a forbidden
	configuration in } S \cap (n^{1/9}, n]\right),
    \end{multline*}
    and Proposition \ref{lem:NoForbidden_v1} follows from Proposition
    \ref{lem:NoForbidden}.
\end{proof}

In order to prove Proposition \ref{lem:NoForbidden} we will make repeated use of
the union bound to control the appearance probability of different
configurations. Concretely, given a family $\cA \subset 2^{[n]}$ of uniform
cardinality $r = |A|$ for every $A \in \cA$, and given a random set $S \sim
[n]_p$, we will use that
\begin{equation}
\label{eq:union_bound}
\Pr\left(\exists A \in \cA \text{ with } A \subset S\right) \leq \sum_{A \in \cA} \Pr(A
\subset S) \leq |\cA| p^{r}.
\end{equation}

Before going into the full details of the proof, let us sketch why the
proposition is true. Suppose that all possible forbidden configurations in $S$
satisfy $f=z=w=1$ and that the nine numbers $a, b, c, d, e, x, y, u, v \in S$
are all distinct, which is in a sense the generic case. Since all of these nine
numbers are divisors of $c$, the total number of configurations of this form is
at most $nD(n)^{9}$. Hence, union
bounding as in \eqref{eq:union_bound} gives that the probability of having such a
configuration in $S$ is at most $nD(n)^9p^{9} = w(n)^{-9}D(n)^9 = o(1)$.

However, this is far from a proper proof. To begin with, the elements $f, z, w$
may not be equal to one, but more importantly, the different terms in a
forbidden configuration may not be distinct, giving worse bounds on the
appearance of a fixed configuration. Therefore, in order to recover the union
bound argument, it is necessary to prove that there are far fewer forbidden
configurations of this kind. Hence, our strategy will be to split the possible
forbidden configurations according to which variables are equal, and bound the total
number of configurations of each kind, to be able to apply the union bound.

It is precisely because of this reason that we require $k \not
\in R\cdot R\cdot R$ in the first condition of Algorithm \ref{alg1}. If we
define the more natural algorithm allowing $k \in R \cdot R \cdot R$, this leads
to the worse threshold of $p \leq n^{-1/8}\omega(n)^{-1}$ for the algorithm to
terminate successfully a.a.s.\ (in particular, there are too many
forbidden configurations of the form $x=d$, $z=w=f=1$, and all other variables
distinct, for union bound arguments to work).

From the previous discussion, we see that, when union bounding over a
certain type of forbidden configuration, the following parameter will be crucial.
\begin{definition}
The \emph{effective size} of a forbidden configuration $M$ is the number of
distinct elements different than $1$ in $M$.
\end{definition}

We now begin in earnest the proof of Proposition \ref{lem:NoForbidden}, which
will need a fair amount of casework. From now on, let $M = (a, b, c, d, e, f, x,
y, z, u, v, w)$ be a possible forbidden configuration satisfying the conditions
in Definition \ref{def:forbidden_config}. In particular, we know that
\begin{equation}
    \label{eq:relation_forbidden}
    xyzuvw=def,
\end{equation}
a key relation from which we will extract information in most cases we consider.
In the setting of Proposition \ref{lem:NoForbidden}, we only consider forbidden
configurations where $x, y, u, v, d, e \geq n^{1/9}$, $a, b, c \geq n^{2/9}$ and
$f, z, w \geq n^{1/9}$ if they are not $1$.
Before going into the actual cases, we also isolate a  simple lemma that is useful for
counting the number of certain configurations.

\begin{lemma}\label{ei}
    Let $e_1,\dots,e_k \in \mathbb{N}$ with $1 \le e_1\le e_2\le \dots\le e_{k}$ 
    and $1\leq t\leq n$. Let $e\coloneq \sum_{i=1}^ke_i$. Then the number of $k$-tuples
    $(a_1,\dots,a_k)$ such that $a_i\in[n]$ for all $i\in[k]$, $ a_i\ge t$ for all $i \in [k]$ with $e_i\ge 2$, and
    $\prod_{i=1}^{k} a_i^{e_i} \leq n$ is at most \[  nt^{k-e}D(n)^k\]

\end{lemma}
\begin{proof}
As $\prod_{i=1}^{k} a_i^{e_i} \le n$, and $a_i\geq t$ for each $i\in [k]$ with $e_i\geq 2$, any valid choice of $(a_1,\ldots,a_k)$ will have $\ell:=\prod_{i=1}^ka_i\leq nt^{k-e}$. After choosing $\ell$ (at most $nt^{k-e}$ choices) the $a_i$ are chosen as divisors of $\ell$ (at most $D(n)$ choices for each $i\in [k]$). 
\end{proof}
In particular, we will apply it several times in the following form, which, in
informal terms, tells us that what we lose by having repetitions in
$x,y,z,u,v,w$ is compensated by better bounds on the number of such
configurations.
\begin{corollary}\label{macro}
    Given a forbidden configuration with effective size $s$, define
    $r = r_1 - r_2$, where $r_1$ is the number of terms among $x,y,z,u,v,w$
    which are not $1$, and $r_2 = \vert \{x,y,z,u,v,w\}\backslash\{1\}\vert$, so
    that $r$ counts the number of repetitions among terms that are not one.
    Forbidden configurations such that $s+r\ge 9$ appear asymptotically almost
    never in $[n]_p \cap (n^{1/9}, n]$.
\end{corollary}
\begin{proof}
    Consider fixed values of $s, r, r_1, r_2$ such that $s+r \geq 9$. Applying
   Lemma \ref{ei} to the condition $xyzuvw \le n$ with $t=n^{1/9}$
    implies there are at most $nn^{(r_2-r_1)/9}D(n)^{r_2} = n ^{1-r/9}
    D(n)^{r_2}$ forbidden configurations with this value of $s, r, r_1, r_2$. Union bounding as in \eqref{eq:union_bound} tells us that these
    configurations appear with probability at most
    $n^{1-r/9} D(n)^{r_2}p^s = o(1)$.
    Again union bounding over all constantly many possible values of $s, r, r_1$
    and $r_2$ yields the result.
\end{proof}

We now split forbidden configurations into different cases, where we use
somewhat different approaches to count the number of forbidden configurations.
Given a forbidden configuration, let $s$ be its effective size, let $m$ count
how many numbers among $d, e, f$ belong to $\left\{ x,y,z,u,v,w \right\}$, and
let $m_d$ and $m_f$ denote the number of terms among $x,y,z,u,v,w$ equal to $d$
and $f$ respectively. In Table \ref{table:cases} we see a description of all
possible cases we consider. 

\begin{table}[h]
\centering
\begin{tabular}{l l l l}
\toprule
\multicolumn{3}{c}{Conditions} & Case \\
\midrule

\multirow{3}{*}{$a=b$}
& & $s \geq 5$   & \ref{case:1.1} \\
& & $s = 4$ &   \ref{case:1.2} \\
& & $s = 3$ &   \ref{case:1.3} \\

\midrule


\multirow{11}{*}{$a\neq b$}
& \multirow{4}{*}{$d, e, f$ distinct} & $m = 0$ & \ref{case:2.1} \\
&                     & $m = 1$& \ref{case:2.2} \\
&                     & $m = 2$ & \ref{case:2.3} \\
&                     & $m = 3$ & \ref{case:2.4} \\

\cmidrule(lr){2-4}

& $d = e= f$ &  & \ref{case:3.1} \\

& $d=e$, $f=1$ &  & \ref{case:3.2} \\

\cmidrule(lr){2-4}

& \multirow{5}{*}{$d=e\neq f, f \neq 1$} & $m_d=0, m_f=0$ & \ref{case:3.3} \\
&                     & $m_d=0, m_f\geq1$ & \ref{case:3.4} \\
&                     & $m_d=1, m_f= 0$ & \ref{case:3.5} \\
&                     & $m_d=1, m_f \geq 1$ & \ref{case:3.6} \\
&                     & $m_d \geq 2$ & \ref{case:3.7} \\
\bottomrule
\end{tabular}
    \caption{Forbidden configuration cases}
    \label{table:cases}
\end{table}
The only possibilities missing from the table, which are when $a\neq b$ and $d=f
\neq e$ or $e=f \neq d$, are symmetric to the case where $a\neq b$ and $d=e\neq
f$, with $f\neq 1$, and can be proved by relabeling the terms.
As the naming suggests, we have divided these cases into three different groups.
As we have already said, for every case we prove that configurations of
that kind appear asymptotically almost never. Union bounding over all $14$ cases
proves Proposition \ref{lem:NoForbidden}.

\subsubsection{Forbidden configurations of Group 1}
All cases in this group satisfy that $a=b$. If $a=b$, then $c=ab=a^2$ must be a
perfect square. Hence, there are only $n^{1/2}$ possible choices for $c$, and
all other terms of the configuration divide $c$, so the total number of
forbidden configurations of this kind is at most $n^{1/2}D(n)^{11}$. We split
the configurations according to their effective size $s$, which must be at least
$3$ because $a, c, d$ are distinct.

\begin{case}{1.1}[$s \geq 5$]
    \label{case:1.1}
    If the effective size of the configuration is at least $5$, a union bound as in
    \eqref{eq:union_bound} bounds the appearance probability for these
    configurations by $n^{1/2}D(n)^{11}p^5 = o(1)$.
\end{case}
\begin{case}{1.2}[$s=4$]
    \label{case:1.2}
If the effective size is four, since $d,e,f,x,y,z,u,v,w$ are
all distinct from $a$ and $c$, they can only take values $1$, $g$ or $h$ for
some variables $ g,h \geq n^{1/9}$. Plugging this into \eqref{eq:relation_forbidden}
gives a relation of the form $g^\ell h^r = g^{i}h^j$ for some $i,j,\ell,r$ with
$\ell+r \in \{4,5,6\}$, $i+j \in \{2, 3\}$. Reordering, we have that $g^{\ell-i}
= h^{j-r}$. As $i+j < l+r$, one of the two sides of the equation is not equal to
one, so in fact both sides cannot be equal to one. This implies that $g$ and $h$
are powers of a common factor, and then $a$ itself must be a perfect power. Once
we have that $a$ is a perfect power, then $c=a^2$ must be a perfect power of
exponent at least 4. Hence, we have $O(n^{1/4})$ possible choices for $c$ (accounting for powers of exponent $4$ and greater) and configurations
of this family appear at most $O(n^{1/4}D(n)^{11})$ times in $[n]$. A union
bound gives that the appearance probability is at most $O(n^{1/4}D(n)^{11}p^4) =
o(1)$.
\end{case}

\begin{case}{1.3}[$s=3$]
    \label{case:1.3}
Finally, if the effective size of one such configuration is three, then $a$ and
$b$ must be squares or cubes themselves, implying that configurations of this
family appear at most $O(n^{1/4}D(n)^{11})$ in $[n]$.
Now a union bound is enough to finish. 
\end{case}

\subsubsection{Forbidden configurations of Group 2}
In this group we study forbidden configurations satisfying
\[
    a \neq b \qquad d, e, f \text{ distinct.}
\]
Recall that $m$ counts how many numbers among
$d,e,f$ belong to $\{x,y,z,u,v,w\}\backslash\{1\}$, and let $r$ be as in
Corollary \ref{macro}. Since $d, e, f$ are distinct and different to $a,b$ (on account of their colours), we have that $s+r+m$ is
precisely the number of non-one terms in the configuration.
We split the configurations according to the value of $m$.

\begin{case}{2.1}[$m=0$]
    \label{case:2.1}
Since the number of non-one terms in the configuration is at least $9$ and
$m=0$, we have that $s+r = s+r+m \geq 9$, and we may apply Corollary \ref{macro}.
\end{case}

\begin{case}{2.2}[$m=1$]
    \label{case:2.2}
Suppose that $f = z = w = 1$. In that case,
\eqref{eq:relation_forbidden} reduces to $de=xyuv$, and the fact that $m=1$ tells
us that we can simplify a factor on each side of the equality, contradicting
condition (iv) in Definition \ref{def:forbidden_config}. 

It follows that if $m=1$ then $f,z,w$ cannot be all equal to one. Therefore, the
number of non-one terms in the forbidden configuration is at least $10$, so $s+r
+1 = s+r + m \geq 10$,  so $s+r \geq 9$ and again we conclude by applying
Corollary \ref{macro}.
\end{case}

\begin{case}{2.3}[$m=2$]
    \label{case:2.3}
Suppose that $f = 1$. Then
\eqref{eq:relation_forbidden} becomes $de=xyzuvw$, and from $m=2$ we may
deduce that the product of four numbers among $u,v,w,x,z,w$ is $1$, a
contradiction. Hence assume $f\ne 1$. The fact that $m=2$ tells us that we can
simplify two terms on each side of the equality $def=xyzuvw$. The resulting
equality will contradict condition (iv) in Definition \ref{def:forbidden_config}
unless $z$ and $w$ are both not equal to one. In that case, the total number of
non-one terms in the configuration is 12, so $s+r \geq 12-m \geq 9$ and we may
apply Corollary \ref{macro}.
\end{case}

\begin{case}{2.4}[$m=3$]
    \label{case:2.4}
This case is actually void, as $def = xyzuvw$ cannot hold if
$d,e,f\in \{x,y,z,u,v,w\}$, because it would imply that a product of three numbers among $x,y,z,u,v,w$ is equal to one.\\
\subsubsection{Forbidden configurations of Group 3}
Finally, it only remains to prove that forbidden configurations where
\[
    a \neq b \qquad d,e,f \text{ not distinct}
\]
occur asymptotically almost never. We begin by proving the first two cases.
\end{case}

\begin{case}{3.1}[$d=e=f$]
    \label{case:3.1}
 In this case, $c=def=d^3$ is a cube. Hence, there are at most
 $n^{1/3}D(n)^{11}$ possible configurations of this form. We also know they have
 effective size at least four, because $a, b, c, d$ are all distinct. Therefore,
 they appear with probability at most $n^{1/3}D(n)^{11}p^4 = o(1)$.
 \end{case}

 \begin{case}{3.2}[$d=e, f=1$]
    \label{case:3.2}
 In this case, $c=def=d^2$ is a perfect square.
 Hence, there are at most $n^{1/2}D(n)^{11}$ configurations of this form in
 $[n]$. By a union bound, the appearance probability of such a configuration of
 effective size at least $5$ is at most $n^{1/2}D(n)^{11}p^5 = o(1)$. As
 $a,b,c,d$ are all different, effective size $4$ is the only remaining
 possibility. However, in such a case, $a$ and $b$ must both be equal to either
 $d^2$ or $d^3$, contradicting $ab = c = d^2$. Therefore, there are no
 configurations of this form with effective size $4$.
 \end{case}

In the remaining cases, we assume without loss of generality that $d=e\neq f$
and $f\neq 1$. Recall that $m_d$ and $m_f$ respectively denote the number of
values among $x,y,z,u,v,w$ equal to $d$ and $f$. We split all these forbidden
configurations into 5 groups  based on the values of $m_d$ and $m_f$. As usual,
we prove that configurations in each case occur asymptotically
almost never.

 \begin{case}{3.3}[$m_d=0, m_f=0$]
    \label{case:3.3}
This case follows from Corollary \ref{macro}. Indeed, adopting the notation there we have that $s=5+r_2=5+r_1-r$ and so $s+r\geq 5+r_1\geq 9$. 
\end{case}

\begin{case}{3.4}[$m_d=0, m_f \geq 1$]
    \label{case:3.4}
    After possibly relabeling terms, we may assume that $x=f$. In that case,
    \eqref{eq:relation_forbidden} becomes
    \[
	fyzuvw = d^2f,
    \]
    which implies that $yzuvw = d^2$. We split our analysis according to the size of $d$.
    \begin{itemize}
	\item If $d \geq n^{3/9}$, from $c=d^2f$ and Lemma \ref{ei}
	    we see that there are at most $n^{6/9}D(n)^2$ possible choices for
	    $c$, and at most $n^{6/9}D(n)^{11}$ possible configurations of this
	    kind. If the effective size is at least six, they appear with
	    probability at most $n^{6/9}D(n)^{11}p^6 = o(1)$. Since $a, b, c,
	    d, f$ are distinct, the only remaining possibility is effective size
	    $5$, and in that case $x, y, z, u, v, w$ are all equal to $1$ or
	    $f$, so $c=xyzuvw$ is a perfect power of exponent at least $4$.
	    Hence, there are at most $O(n^{1/4})$ choices for $c$, and union
	    bounding again concludes.
	\item If $d < n^{3/9}$, we study the probability that the
	    subconfiguration formed by $(y, z, u, v, w, d, b)$ appears. By
	    assumption, we have at most $n^{3/9}$ possible choices for $d$. Since $y, z,
	    u, v, w$ all divide $d^2$, once we fix $d$ there are at most $D(n)^{5}$ choices for
	    these terms, and then $b=uvw$ is uniquely determined. Hence, there
	    are at most $n^{3/9}D(n)^{5}$ possible subconfigurations in
	    $(n^{1/9}, n]$, and every subconfiguration has effective size at
	    least three, because $b, d, u$ are distinct, so the probability that such a subconfiguration appears is bounded by $n^{3/9}D(n)^5p^3 = o(1)$. This also bounds the probability of the whole configuration appearing.
    \end{itemize}
\end{case}

 \begin{case}{3.5}[$m_d=1, m_f=0$]
    \label{case:3.5}
Assume without loss of generality that $u = d$. Note that
\eqref{eq:relation_forbidden} gives
\[
    xyzvw = df = \frac{d^2f}{d} \le n^{8/9},
\]
since $d \geq n^{1/9}$.
Fixing some choice $r'_1$ to be the number of terms among $x,y,z,v,w$ that are not equal to 1 and a choice of $r_2'=|\{x,y,z,v,w\}\setminus \{1\}|$, we have from Lemma \ref{ei} (applied with $n$ replaced by $n^{8/9}$) that there are at most $n^{8/9}n^{(r'_2-r'_1)/9}D(n)^{r_2'}$ choices for $x,y,z,v$ and $w$. Furthermore, note that once $x,y,z,v,w$ are given, there are at most $D(n)^{2}$
ways to complete the configuration, since there will be at most $D(n)^2$ choices
for $d$ and $f$, and $d,f$ determine $c$. The effective size of such a configuration is precisely $s=5+r_2'$ with the 5 accounting for $a,b,c,d$ and $f$. Therefore the probability that a configuration with values $r_1'$ and $r_2'$ appears is at most $D(n)^{r_2'+2}n^{(8+r_2'-r_1')/9}p^{s}=o(n^{3-r_1'}/9)=o(1)$ as $r_1'\geq 3$. Union bounding over all choices of $r_1'$ and $r_2'$ thus a.a.s.\ rules out all configurations in this group.
\end{case}

\begin{case}{3.6}[$m_d=1, m_f\geq 1$] In this case,
    \label{case:3.6}
    \eqref{eq:relation_forbidden} gives
    $d^2f = xyzuvw$, and we know that a $d$ and an $f$ in the left hand side will
cancel out. Thus $d$ must be the product of two, three or four numbers different
from $1$ among $x,y,z,u,v,w$. The first two cases are impossible by condition
(iv) in Definition \ref{def:forbidden_config}, so $d$ is the product of four
such numbers. Therefore $d \ge n^{4/9}$ and $f \leq n/d^2 \leq n^{1/9}$, which cannot occur.
\end{case}

\begin{case}{3.7}[$m_d\geq 2$] In this case,
    \label{case:3.7}
    \eqref{eq:relation_forbidden} gives again $d^2f = xyzuvw$.
     Since $m_d \ge 2$, we deduce that $f$ must be the product of two, three or
     four numbers among $x,y,z,u,v,w$, with the other two being equal to $d$.
     Thus by condition (iv) in Definition \ref{def:forbidden_config}, $f$ must
     be the product of four such numbers, implying that $z,w\ne 1$ and $m_f=0$.
     Now  using that $z,w\ne 1$, we have that the $r$ in Corollary \ref{macro} counts all of the repetitions among $x,y,z,u,v,w$. 
     We then have that the effective size of the configuration is $s=10-r$ as $a,b,c$ and $f$  are all distinct and none of the terms $x,y,z,u,v,w$ coincide with $a,b,c$ or $f$ (those that are not equal to $d$ divide $f$ and so cannot equal $f$). The conclusion thus follows from Corollary \ref{macro}. 
     \qed  
\end{case}

\section{The (1)-statement for two colours} \label{sec:1-statement}
In this section, we prove the (1)-statement of Theorem
\ref{thm: 2 col}. To do so, we find a certain small pattern that is
product-Schur, and then show that a copy of this pattern appears in $[n]_p$ with
high probability when $p$ is large enough. We begin by proving that the desired
pattern is product-Schur.
\begin{lemma}
    Given $a, b, c, d \in [n]$, the set
    \begin{equation}
\label{eq:schur_pattern}
	S_{a,b,c,d} = \{a, b, ab, c, ac, abc, a^2bc, a^2b^2c, d, ad, bd, abd, a^2d, a^2bd\}
    \end{equation}
    is product-Schur.
\end{lemma}
\begin{proof}
    Suppose we have a colouring $S_{a, b, c,d} = B \cup R$ with no monochromatic
    products. Without loss of generality, assume that $a \in B$.
    
    We claim that if $b \in B$, the subpattern
    \[\{a, b, ab, c, ac, abc, a^2bc, a^2b^2c\},\]
    must have a monochromatic product. To see this, we split depending on the
    colour of $c$:
    \begin{itemize}
	\item If $c \in B$, from the lack of monochromatic products we deduce
	    that
	    \[
		a, b, c \in B \implies ab, ac \in R \implies (ab)(ac)=a^2bc
		    \in B,
	    \]
	    and then
	    \[
	    a^2bc \in B, a \in B \implies (a^2bc)/a = abc \in R.
	    \]
	    From $(abc)(ab) = a^2b^2c$ we deduce that $a^2b^2c \in B$, but then
	    $b(a^2bc) = a^2b^2c$ is a solution with all terms belonging to $B$,
	    a contradiction.
	\item If $c \in R$, we see that
	    \[
		a, b \in B \implies ab \in R,
	    \]
	    which together with $c \in R$ gives that $(ab)c =abc 
	    \in B$. Therefore,
	    \[
		a, b, abc, \in B \implies a(abc) = a^2bc \in R, (abc)/b = ac \in
		R,
	    \]
	    a contradiction because $(ab)(ac) = a^2bc$ is a monochromatic
	    product.
    \end{itemize}

    If $b \in R$, the subpattern
    \[\{a, b, ab, d, ad, bd, abd, a^2d, a^2bd\}\]
    must have a monochromatic product by an analogous reasoning, where fixing the
    colour of $d$ forces the rest of the colouring, leading to a contradiction
    in both cases.
\end{proof}
The (1)-statement of Theorem \ref{thm: 2 col} is then a
consequence of the following proposition.
\begin{proposition}
    For $p  = \omega(n^{-1/11})$, the random set $[n]_p$ a.a.s.\ contains a subset of the form
    \eqref{eq:schur_pattern}.
\end{proposition}
\begin{proof}
    For convenience, we will only consider sets $S_{a,b,c,d}$ for integers $a \in A$, $b \in B$, $c \in C$, $d\in D$, where 
\begin{align*}
A &= [(1-\delta)n^{1/11}, n^{1/11}],\\
B &= [(1-\delta)n^{2/11}/2, n^{2/11}/2],\\
C &= [(1-\delta)n^{5/11}, n^{5/11}],\\
D &= [(1-\delta)n^{7/11}/3, n^{7/11}/3].
\end{align*}
Here $\delta$ is any small enough fixed positive constant, say, $\delta =
10^{-3}$. These intervals are chosen so that the different possible terms of
$S_{a, b, c, d}$ lie in disjoint intervals. Concretely, we have that
\begin{multicols}{2}

\noindent
\begin{align*}
    a &\in A = [(1-\delta)n^{1/11}, n^{1/11}], \\
b &\in B = [(1-\delta)n^{2/11}/2, n^{2/11}/2], \\
ab &\in I_{ab} \coloneq [(1-2\delta)n^{3/11}/2, n^{3/11}/2], \\
c &\in C = [(1-\delta)n^{5/11}, n^{5/11}], \\
ac &\in I_{ac} \coloneq [(1-2\delta)n^{6/11}, n^{6/11}], \\
abc &\in I_{abc} \coloneq [(1-3\delta)n^{8/11}/2, n^{8/11}/2], \\
a^2bc &\in I_{a^2bc} \coloneq [(1-4\delta)n^{9/11}/2, n^{9/11}/2],
\end{align*}
\columnbreak
\noindent
\begin{align*}
    a^2b^2c &\in I_{a^2b^2c} \coloneq [(1-5\delta)n^{11/11}/4, n^{11/11}/4], \\
d &\in D= [(1-\delta)n^{7/11}/3, n^{7/11}/3], \\
ad &\in I_{ad} \coloneq [(1-2\delta)n^{8/11}/3, n^{8/11}/3], \\
bd &\in I_{bd} \coloneq [(1-2\delta)n^{9/11}/6, n^{9/11}/6], \\
abd &\in I_{abd} \coloneq [(1-3\delta)n^{10/11}/6, n^{10/11}/6], \\
a^2d &\in I_{a^2d} \coloneq [(1-3\delta)n^{9/11}/3, n^{9/11}/3], \\
a^2bd &\in I_{a^2bd} \coloneq [(1-4\delta)n^{11/11}/6, n^{11/11}/6].
\end{align*}
\end{multicols}
\vspace{-2\baselineskip} 
\noindent which are disjoint intervals for any large enough $n$. 


We now prove that if $p = \omega(n^{-1/11})$, we
can find at least one instance of $S_{a,b,c,d}$ in $[n]_p$. We will do this by
revealing $[n]_p$ step by step, and choosing one of $a,b,c,d$ at each step.
\begin{itemize}
    \item In the first step, we reveal $A$. By independence, we have that
    \[
	\Pr\left( A \cap [n]_p = \varnothing\right) =
	\prod_{i\in A} \Pr(i \not \in [n]_p) = (1-p)^{\delta n^{1/11}} 
	\le e^{-\delta pn^{1/11}} = o(1).
    \]
    If $[n]_p \cap A \neq \varnothing$, we set $a$ to be
    any element in $[n]_p \cap A$ and proceed to the second step.
\item In the second step, we reveal $B$ and $I_{ab}$, which are independent from
    $A$ by disjointness. With an analogous
    reasoning, we see that
    \[
	\Pr\Big(\nexists i \in B\cap[n]_p \text{ with } ai \in 
	[n]_p\Big) = \prod_{i \in B} \Pr(i \not \in B \text{ or } ai \not \in
	    [n]_p) \leq e^{-\delta p^2 n^{2/11}} = o(1).
    \]
    If there exists such an $i$, we set $b$ to be equal to it and carry out the
    third step.
\item In the third step, we reveal $C, I_{ac}, I_{abc}, I_{a^2bc}$ and
    $I_{a^2b^2c}$. Again, we have
    \[
	\Pr\Big(\nexists i \in C\cap[n]_p \text{ with } ai, abi, a^2bi, a^2b^2i
	\in [n]_p\Big) \leq e^{-\delta p^5 n^{5/11}} = o(1),
    \]
    and we let $c$ be such an $i$ if it exists. If there is one, we move to the
    final step.
\item In the last step, we reveal $D, I_{ad}, I_{bd}, I_{abd}, I_{a^2d}$ and
    $I_{a^2bd}$. Again,
    \[
	\Pr\Big(\nexists i \in D\cap[n]_p \text{ with } ai, bi, abi, a^2i, a^2bi
	\in [n]_p\Big) \leq e^{-\delta p^6 n^{7/11}} = o(1),
    \]
    and let $d$ be such an $i$ if it exists.
\end{itemize}
Union bounding over the complement,  we see that the probability that we may not
find such $a, b, c, d$ is $o(1)$. Therefore, a.a.s.\ there exists a set of the form
$S_{a, b, c, d}$ in $[n]_p$ as required.
\end{proof}

\section{More colours} \label{sec:more cols}
We now look at an arbitrary number of colours $r \geq 2$. Recall 
the (0)-statement in Theorem \ref{thm: more col}.
\begin{theorem}
    For $p = o\left(n^{-1/S'(r)}\right)$ the set $[n]_p$ is $r$-product-Schur asymptotically almost never.
\end{theorem}
\begin{proof}
    To begin with, note that
    \[
	\Pr\left([n]_p \cap [n^{1/S'(r)}] = \varnothing\right) =
	(1-p)^{n^{1/S'(r)}} \geq e^{-2pn^{1/S'(r)}} = 1-o(1).
    \]
    Therefore,
    \[
	\Pr\left([n]_p \cap [n^{1/S'(r)}] \neq \varnothing\right) = o(1).
    \]

Now let $\psi \colon [S(r)-1]\to [r]$ be some $r$-colouring of $[S'(r)-1]$
avoiding monochromatic sums and monochromatic shifted sums (solutions to
$a+b=c-1$). We colour the intervals of the form $I_i=(n^{i/S'(r)},
n^{(i+1)/S'(r)}]$ for $i\in [S'(r)-1]$ according to $\psi$. More concretely, we
colour all numbers in $I_i \cap [n]_p$ with the colour $\psi(i)$. Since a.a.s.\ all numbers in $[n]_p$ belong to one of these
intervals, all numbers end up coloured a.a.s..

Finally, we observe that, if $a\in I_\alpha$ and  $b\in I_\beta$, then $ab
\in I_{\alpha+\beta} \cup I_{\alpha+\beta+1}$. Therefore, if $a,b,ab$ have the
same colour, then $\psi(\alpha) = \psi(\beta)$ and $\psi(\alpha), \psi(\beta) \in \{ \psi(\alpha+\beta),
\psi(\alpha+\beta+1)\}$, which is impossible by the definition of $\psi$. This
proves that our colouring contains no monochromatic solutions to $xy=z$ in
$[n]_p$. We conclude by noting that this is a full colouring of $[n]_p$ a.a.s..
\end{proof}

\begin{remark}\label{rmrk}
    For small values of $r$ we can in fact slightly improve the bound on the
    threshold with a variation of the idea. Notice that colouring $I_1$ and $I_2$ with the same colour does not generate monochromatic tuples, since a.a.s.\ no monochromatic tuples exist in $[p^{-3+o(1)}]_p$ by a union bound argument. 

    This trick allows for colourings $\psi$ that avoid all triples of the forms $a+b=c$ and $a+b+1=c$ except $1+1=2$. Computer calculations show that with $3$ colours we can colour up to\footnote{The construction here is  AABBBACCCACCCABBBA.} $18$ (while $S'(3)=13$) and with four colours up to\footnote{The construction here is  AABBBACCCACCCABBBADDDADDDABBBADDDADDDABBBACCCACCCABBBA.} $54$ (when $S'(4)=40$). This improves the bounds for $r\in\{3,4\}$ to $n^{-1/19}$ and $n^{-1/55}$, respectively. However, we have not found a better general construction using this idea.
\end{remark}


    Finally, we prove the (1)-statement in Theorem \ref{thm: more col}. The idea of the
    proof is the same we used for the 1-statement of the two colour case, but
    with a more general pattern that applies to all $r$. We recall the statement.
\begin{theorem}
    For $p = \omega(n^{-1/S(r)^2})$, the set $[n]_p$ is $r$-product-Schur a.a.s..
\end{theorem}
\begin{proof}
Consider the pattern
\[
U_x = \{x,x^2,\ldots,x^{S(r)}\}
\]
for $2\le x \le n^{1/S(r)}$. Note that $U_x$ is not $r$-product-colourable by
the definition of $S(r)$. We now show that, for our range of $p$, a.a.s.\ there is
$U_x \subset [n]_p$ for some $x \in [n]_p$.

Consider the family of sets $\{U_x \mid x\in X\}$ for the set $X\subset
[n^{1/S(r)}]$ built greedily in the following manner. Iterate through values of $x\in [n^{1/S(r)}]$ and add $U_x$ if it is disjoint from all other $U_y$ for $y$ already belonging to the set. For any given $x \in [n^{1/S(r)}]$, there are at most $S(r)^2$ values $y \in [n^{1/S(r)}]$ such that $U_x \cap U_y \neq \varnothing$, so $X$ has size at least $\frac{1}{S(r)^2} n^{1/S(r)}$.

Note that for any given $x$ we have $\mathbb P(U_x\subset [n]_p) = p^{S(r)}$,
and that the events $U_x\subset [n]_p$ for $x\in X$ are mutually independent on
account of the sets $U_x$ being disjoint. Hence, the probability that $U_x \not
\subset [n]_p$ for every $x \in X$ satisfies
\begin{align*}
\mathbb P\left(\bigwedge_{x\in X} U_x\not\in [n]_p\right) &= \left(1-p^{S(r)}\right)^{\vert X\vert}\\
&\le 
\left(1-p^{S(r)}\right)^{\frac{1}{S(r)^2} n^{1/S(r)}} \\
&\le \exp\left(-p^{S(r)}\frac1{S(r)^2} n^{1/S(r)}\right) = o(1).
\end{align*}
Hence, $[n]_p$ contains $U_x$ for some $x\in[n]_p$ a.a.s., and such a pattern is $r$-product-Schur, so $[n]_p$ also is.
\end{proof}

\section{Concluding remarks} \label{sec:conclude}

In this paper, we showed that the threshold for the random set $[n]_p$ lies between $n^{-1/9-o(1)}$ and $n^{-1/11}$. We believe that our upper bound provides the true location of the threshold. 

\begin{conjecture} \label{conj:2col}
    The threshold for $[n]_p$ to be product-Schur lies at $n^{-1/11}$. 
\end{conjecture}

This conjecture is motivated by  two sub-conjectures, each of which could provide natural steps towards its resolution. Firstly, we believe that the threshold is determined by the appearance in $[n]_p$ of some bounded size configuration which is itself product-Schur. Secondly, we believe that the configuration used in our proof of Theorem \ref{thm: 2 col} (1) is optimal, that is, there is no other bounded size configuration that is product-Schur and appears at a density less than $p=n^{-1/11}$. 
As a first step towards Conjecture \ref{conj:2col}, it would be interesting to improve on the lower bound of $n^{-1/9-o(1)}$. We believe that it might be possible   to slightly improve the lower bound using a  more involved greedy colouring scheme, but getting a significant improvement on the lower bound would probably require new ideas. 

For the $r$-product-Schur problem with more colours, we showed that the threshold $p_r^*$ lies between $n^{-S'(r)}$ and $n^{-S(r)^2}$ where $S(r)$ is the Schur number and $S'(r)$ is the double-sum Schur number. We remark that one could also try to use the idea for our $2$-colour (0)-statement to get a lower bound in the $r$-colour case. Indeed, the idea there was to take a greedy colouring and show that it a.a.s.\ does not fail to colour all the integers. One can show by induction that there are $3^r$ integers in a (non-degenerate) configuration which witnesses the failure of the algorithm and so one cannot expect a better bound than $p_r^*\geq n^{-1/3^r}$ using this approach. As
$S'(r)\geq 3S(r-1)-1=\Omega(3.17176^r)$  \cite{AbbottHanson,SchurLowerBound}, this does worse than the lower bound we give here for $r$ sufficiently large. 

It would be very interesting to close the gap on our knowledge of $\log_n{1/p_r^*}$ which we know is between exponential and factorial in $r$. Given that this is also the limit of our knowledge on the growth of Schur numbers \cite{SchurLowerBound,SchurUpperBound}, this seems like it could be a considerable challenge. Nonetheless, it may be possible to improve the upper and lower bounds of $\log_n{1/p_r^*}$ in terms of the Schur number $S(r)$. Finally we remark that it would be natural to study other nonlinear equations in this random Ramsey setting and we believe our techniques developed here will be useful for this. In particular, some pertinent candidates are  $a^2=bc$  (with $b\neq c$) 
and $ab=cd$ (with $\{a,b\}\neq \{c,d\}$).

\subsection*{Acknowledgements}
Part of this work was done during  the  \textit{Barcelona Introduction to Mathematical Research -- (BIMR 2024)}   and we would like to thank BGSMath and  Centre de Recerca Matem\`atica (CRM) for organising this summer programme. We also thank Shagnik Das, Letícia Mattos and Juanjo Ru\'e for interesting discussions around this topic, and Jonathan Chapman for helpful comments on a previous version of the manuscript.

\bibliographystyle{amsplain} 
\bibliography{refs} 
\end{document}